\documentclass[]{article}
\usepackage{amsmath,amsthm}
\usepackage{amssymb}
\usepackage{amsfonts}
\usepackage{subfigure}
\usepackage{amsmath,mathrsfs,bm}
\usepackage{array}
\usepackage{color,soul}
\usepackage{overpic}
\usepackage{ctable}
\usepackage{tikz-cd}
\usepackage{mathtools}
\usepackage{multirow}
\usepackage{authblk}

\newtheorem{theorem}{Theorem}[section]
\newtheorem{lemma}[theorem]{Lemma}
\newtheorem{definition}[theorem]{Definition}
\newtheorem{remark}[theorem]{Remark}

\numberwithin{equation}{section}
\newcommand{\normmm}[1]{{\left\vert\kern-0.25ex\left\vert\kern-0.25ex\left\vert #1
    \right\vert\kern-0.25ex\right\vert\kern-0.25ex\right\vert}}

\newcommand{\xr}{\widetilde{x}} \newcommand{\yr}{\widetilde{y}}

\newcommand{\curl}{\bm{\mathrm{curl}}}

\begin{document}
\title{A nodal type polynomial finite element exact sequence over quadrilaterals
\thanks{This project is supported by NNSFC (Nos.~61733002, 61572096, 61432003, 61720106005, 61502107) 
and ``the Fundamental Research Funds for the Central Universities''. }}

\author[a]{Xinchen Zhou\thanks{Corresponding author: dasazxc@gmail.com}}
\author[b]{Zhaoliang Meng}
\author[c]{Xin Fan}
\author[b,c]{Zhongxuan Luo}

\affil[a]{\it \small Faculty of Electronic Information and Electrical Engineering, Dalian University of Technology, Dalian 116024, China}
\affil[b]{\it \small School of Mathematical Sciences, Dalian
University of Technology, Dalian, 116024, China}
\affil[c]{\it School of Software, Dalian
University of Technology, Dalian, 116620, China}


\maketitle
\begin{abstract}
This work proposes two nodal type nonconforming finite elements over convex quadrilaterals,
which are parts of a finite element exact sequence.
Both elements are of 12 degrees of freedom (DoFs) with polynomial shape function spaces selected.
The first one is designed for fourth order elliptic singular perturbation problems,
and the other works for Brinkman problems.
Numerical examples are also provided.
\\[6pt]
\textbf{Keywords:} Nodal type; polynomial finite element; exact sequence; quadrilateral meshes.
\end{abstract}

\section{Introduction}

Let $\Omega\subset\mathbb{R}^2$ be a simply connected Lipschitz domain.
The de Rham complex determined by the following exact sequence
\begin{equation}
\label{e: countinuous Stokes complex}
\begin{tikzcd}[column sep=large, row sep=large]
0 \arrow{r} & H^2(\Omega) \arrow{r}{\mbox{\textbf{curl}}}
& \left[H^1(\Omega)\right]^2 \arrow{r}{\mbox{div}}& L^2(\Omega)\arrow{r}&0,
\end{tikzcd}
\end{equation}
also known as the Stokes complex,
is well understood and widely applied in the analysis for many problems in solid and fluid mechanics.
Typical model problems are biharmonic and Stokes problems,
whose solutions can be efficiently approximated by suitable finite element methods.
In particular, a divergence-free Stokes element often plays a role in a certain discretization of
(\ref{e: countinuous Stokes complex}) with some biharmonic element.
A comprehensive review on this topic can be found in \cite{John2016}.
Generally speaking, there are three types of finite element exact sequences approximating (\ref{e: countinuous Stokes complex}).
The first type are completely conforming,
namely, all their components are subspaces of the corresponding forms in (\ref{e: countinuous Stokes complex}).
Typical examples include the sequences derived from the Argyris element \cite{Falk2013},
the singular Zienkiewicz element \cite{Guzman2014},
the Bogner-Fox-Schmit element \cite{Neilan2016}
and the family from spline or macroelements \cite{Christiansen2016b, Fu2018, Arnold1992}, etc.
The second type are semi-conforming, that is, their $0$-forms are $H^2$-nonconforming but $H^1$-conforming,
and their $1$-forms are $H^1$-nonconforming but $H(\mathrm{div})$-conforming.
The sequence constructed via the modified Morley element \cite{Nilssen2001,Mardal2002}
and its higher order extension \cite{Guzman2012} are of this type.
The rectangular Adini element was also recently adopted to formulate an exact sequence \cite{Gillette2018}
as well as the modified nonconforming Zienkiewicz element \cite{Wang2007} on triangles.
The third type are completely nonconforming.
Perhaps the simplest example is the Morley-Crouzeix-Raviart sequence \cite{Morley1968, Crouzeix1973},
whose higher order extension has been recently discovered in \cite{Zhang2018}.
Again this construction has also been extended to the rectangular case \cite{Wang2013,Zhang2009}.
Although all the three types are successful for the discretization of (\ref{e: countinuous Stokes complex}),
for fourth order elliptic singular perturbation problems and Brinkman problems for porous media flow,
only the first two types are sufficiently regular,
while the third type might fail if the mesh is not regular and symmetric enough.
In such a case, the required finite element sequence must approximate not only (\ref{e: countinuous Stokes complex})
but also the following de Rham complex
\begin{equation}
\label{e: de Rham complex}
\begin{tikzcd}[column sep=large, row sep=large]
0 \arrow{r} & H^1(\Omega) \arrow{r}{\mbox{\textbf{curl}}}
& \bm{H}(\mathrm{div};\Omega)\arrow{r}{\mbox{div}}& L^2(\Omega)\arrow{r}&0.
\end{tikzcd}
\end{equation}
Indeed, the modification \cite{Nilssen2001,Mardal2002} for the Morley-Crouzeix-Raviart sequence is a compromise for this dilemma.

Note that all the aforementioned examples are designed for triangular or rectangular meshes.
However, there are fewer researches on the approximation for (\ref{e: countinuous Stokes complex})
and (\ref{e: de Rham complex}) over general convex quadrilaterals,
on which we will give a brief review.
For the first type approximation for (\ref{e: countinuous Stokes complex}),
the $H^2$-conforming Fraijes de Veubeke-Sander element
\cite{Ciavaldini1974, Verbeke1968} is a successful candidate for biharmonic problems.
Owing to a normal aggregation trick,
a subspace method, namely, the reduced Fraijes de Veubeke-Sander element was designed \cite{Ciarlet1978}.
For $H^1$-conforming approximation of the incompressible flow,
Neilan and Sap \cite{Neilan2018} introduced a divergence-free Stokes element from the Fraijes de Veubeke-Sander element.
As far as the second type approximation is concerned,
Bao et al.~\cite{Bao2018} proposed a $H^1$-conforming element for fourth order singular perturbation problems
by enriching a spline element space by bubble functions.
Note that all these elements are spline-based, and so a cell-refinement procedure cannot be avoided.
Comparatively, polynomial shape functions are simple to represent and easy to compute,
and therefore they are often more preferred.
This has been taken into consideration for the third type approximation.
Utilizing the Park-Sheen biharmonic element \cite{Park2013},
Zhang \cite{Zhang2016} generalized the Morley-Crouzeix-Raviart sequence to general quadrilaterals,
but again there is no evidence of its ability to approximate (\ref{e: de Rham complex}).
Recently, a polynomial modification was proposed by Zhou et al.~\cite{Zhou2018},
which works for both (\ref{e: countinuous Stokes complex}) and (\ref{e: de Rham complex}).
We must point out that, the Adini complex \cite{Gillette2018} and the rectangular Morley complex \cite{Wang2013,Zhang2009}
are also successful for the discretization of both (\ref{e: countinuous Stokes complex}) and (\ref{e: de Rham complex}),
but their convergence severely relies on the regularity and symmetry of the rectangular cell,
and therefore cannot be generalized to arbitrary convex quadrilaterals in a obvious manner.
Moreover, we discover that the number of global DoFs of the reduced Fraijes de Veubeke-Sander element \cite{Ciarlet1978} is
significantly less than those of the semi-conforming \cite{Bao2018} and completely nonconforming counterparts \cite{Zhou2018}
benefitting from the nodal type structure.

This work devotes to the construction of a nonconforming finite element exact sequence 
on general convex quadrilateral meshes for approximating both
(\ref{e: countinuous Stokes complex}) and (\ref{e: de Rham complex}),
enjoying the advantages that the elements therein are of nodal type structure,
and their shape functions are polynomials.
In fact, the $0$-form dealing with fourth order elliptic singular perturbation problems
is, in a pseudo $H^1$-conforming manner with respect to (\ref{e: de Rham complex}),
a direct generalization of the modified nonconforming Zienkiewicz element \cite{Wang2007} due to Wang, Shi and Xu.
The DoFs are values and gradients of at vertices.
For the $1$-form designed for Brinkman problems,
we select vertex values and edge normal means as the DoFs.
Optimal and uniform error estimates are also given for both elements with respect to their associated model problems.
From some numerical tests,
one can observe that the performances of both elements are consistent with our theoretical findings.  

The rest of this work is arranged as follows.
In Section \ref{s: 0-form},
the nonconforming finite element working for fourth order elliptic singular perturbation problems 
is defined on quadrilateral meshes. 
Section \ref{s: 1-form} introduces the element for Brinkman problems,
and shows that both the two elements are parts of a finite element exact sequence.
Numerical examples are given in Section \ref{s: numerical examples} to verify the theoretical analysis. 

Throughout the work,
standard notations in Sobolev spaces are adopted.
For a domain $D\subset\mathbb{R}^2$,
$\bm{n}$ and $\bm{t}$ will be the unit outward normal and tangent vectors on $\partial D$, respectively.
The notation $P_k(D)$ denotes the usual polynomial space over $D$ of degree no more than $k$.
The norms and semi-norms of order $m$ in the Sobolev spaces $H^m(D)$
are indicated by $\|\cdot\|_{m,D}$ and $|\cdot|_{m,D}$, respectively.
The space $H_0^m(D)$ is the closure in $H^m(D)$ of $C_0^{\infty}(D)$.
We also adopt the convention that $L^2(D):=H^0(D)$,
where the inner-product is denoted by $(\cdot,\cdot)_D$.
These notations of norms, semi-norms and inner-products also work for vector- and matrix-valued Sobolev spaces,
where the subscript $\Omega$ will be omitted if the domain $D=\Omega$.
Moreover, the positive constant $C$ independent of the mesh size $h$ and parameters $\varepsilon$, $\nu$ and $\alpha$
in the model problems might be different in different places.

\section{Finite element for fourth order elliptic singular perturbation problems}
\label{s: 0-form}

\subsection{Notations of a quadrilateral and an auxiliary affine transformation}

Let $K$ be an arbitrary convex quadrilateral.
The four vertices of $K$ are given by $V_1$, $V_2$, $V_3$, $V_4$ in a counterclockwise order,
and the $i$th edge of $K$ is denoted by $E_i=V_iV_{i+1}$,
whose equation is written as $l_i(x,y)=0$, $i=1,2,3,4$.
Here and throughout the paper, the index $i$ is taken modulo four.
For each $E_{i}$, $M_{i}$ denotes its midpoint,
and $\bm{n}_i$ and $\bm{t}_i$ will be its unit normal and tangential directions.
The equations of lines through $M_{1}M_{3}$,
$M_{2}M_{4}$, $V_{1}V_{3}$ and $V_{2}V_{4}$ read as $m_{13}(x,y)=0$,
$m_{24}(x,y)=0$, $l_{13}(x,y)=0$ and $l_{24}(x,y)=0$, respectively.
Moreover, we assume that all the aforementioned line equations are uniquely determined by
\begin{equation}
\label{e: line unique}
l_1(M_3)=l_2(M_4)=l_3(M_1)=l_4(M_2)=m_{13}(M_2)=m_{24}(M_3)=l_{13}(V_4)=l_{24}(V_3)=1.
\end{equation}

In order to describe the construction,
we recall an auxiliary affine transformation for each $K$
generated by decomposing the standard bilinear mapping (see also \cite{Park2013, Dubach2009, Zhou2016}).
The reference square $\widehat{K}=[-1,1]^2$ is determined by its vertices
$\widehat{V}_1=(-1,-1)^T,\widehat{V}_2=(1,-1)^T,\widehat{V}_3=(1,1)^T$
and $\widehat{V}_4=(-1,1)^T$.
The bilinear mapping $F_K:\,\widehat{K}\rightarrow K$ such that $\widehat{V}_i$ is mapped into
$V_i$ for each $i$ can be decomposed as $F_K=A_K\circ S_K$ with $A_K:\,\widetilde{K}\rightarrow K$ and $S_K:\,\widehat{K}\rightarrow \widetilde{K}$ defined by
\[
\label{equation: affine and simple bilinear transformation}
A_K(\widetilde{\bm{x}})=\bm{A}\widetilde{\bm{x}}+\bm{b},
~S_K(\widehat{\bm{x}})=\widehat{\bm{x}}+\widehat{x}\widehat{y}\bm{s},
~\widetilde{\bm{x}}=(\widetilde{x},\widetilde{y})^T\in\widetilde{K},
~\widehat{\bm{x}}=(\widehat{x},\widehat{y})^T\in\widehat{K},
\]
where $\bm{A}$ is a $2\times 2$ matrix, and $\bm{b},\bm{d}$ and $\bm{s}$ are
two-dimensional vectors given by
\begin{equation}
\label{e: refer para}
\begin{aligned}
\bm{A}&=\frac{1}{4}(V_{3}-V_{4}-V_{1}+V_{2},V_{3}+V_{4}-V_{1}-V_{2}),~\bm{d}=\frac{1}{4}(V_{3}-V_{4}+V_{1}-V_{2}),\\
\bm{b}&=\frac{1}{4}(V_{3}+V_{4}+V_{1}+V_{2}),~\bm{s}=(s_1,s_2)^T=\bm{A}^{-1}\bm{d}.
\end{aligned}
\end{equation}
Figure \ref{fig: trans} gives an example of the intermediate reference element $\widetilde{K}$
and the auxiliary affine transformation $A_K$.
We shall denote a point on $\widetilde{K}$ by $\widetilde{V}=(\widetilde{x},\widetilde{y})^T$
if it equals $A_K^{-1}(V)$ for a point $V=(x,y)^T$ on $K$,
and an edge in $\widetilde{K}$ by $\widetilde{E}$ if it equals $A_K^{-1}(E)$ for an edge $E$ in $K$.
Note that $\widetilde{M}_i$ is also the midpoint of $\widehat{V}_i\widehat{V}_{i+1}$
and $\widetilde{V}_i=\widehat{V}_i+(-1)^{(i+1)}\bm{s}$.
Furthermore, since $K$ is convex, one must have
\begin{equation}
\label{equation: s1+s2<1} |s_1|+|s_2|<1.
\end{equation}
Similarly, we write the function $\widetilde{f}=f\circ A_K$ defined over $\widetilde{K}$
for a function $f$ over $K$.
A simple calculation will derive
\begin{equation}
\label{e: line equations}
\begin{aligned}
\widetilde{l}_1&=\frac{1}{2}\left(-\frac{s_2}{s_1-1}\xr+\yr+1\right),
~\widetilde{l}_2=\frac{1}{2}\left(-\xr+\frac{s_1}{s_2+1}\yr+1\right),
~\widetilde{l}_3=\frac{1}{2}\left(\frac{s_2}{s_1+1}\xr-\yr+1\right),\\
\widetilde{l}_4&=\frac{1}{2}\left(\xr-\frac{s_1}{s_2-1}\yr+1\right),
~\widetilde{l}_{13}=\frac{-\xr+\yr+s_1-s_2}{2(s_1-s_2+1)},
~\widetilde{l}_{24}=\frac{\xr+\yr+s_1+s_2}{2(s_1+s_2+1)},
~\widetilde{m}_{13}=\xr,
~\widetilde{m}_{24}=\yr.
\end{aligned}
\end{equation}

\begin{figure}[!htb]
\centering
\begin{overpic}[scale=0.45]{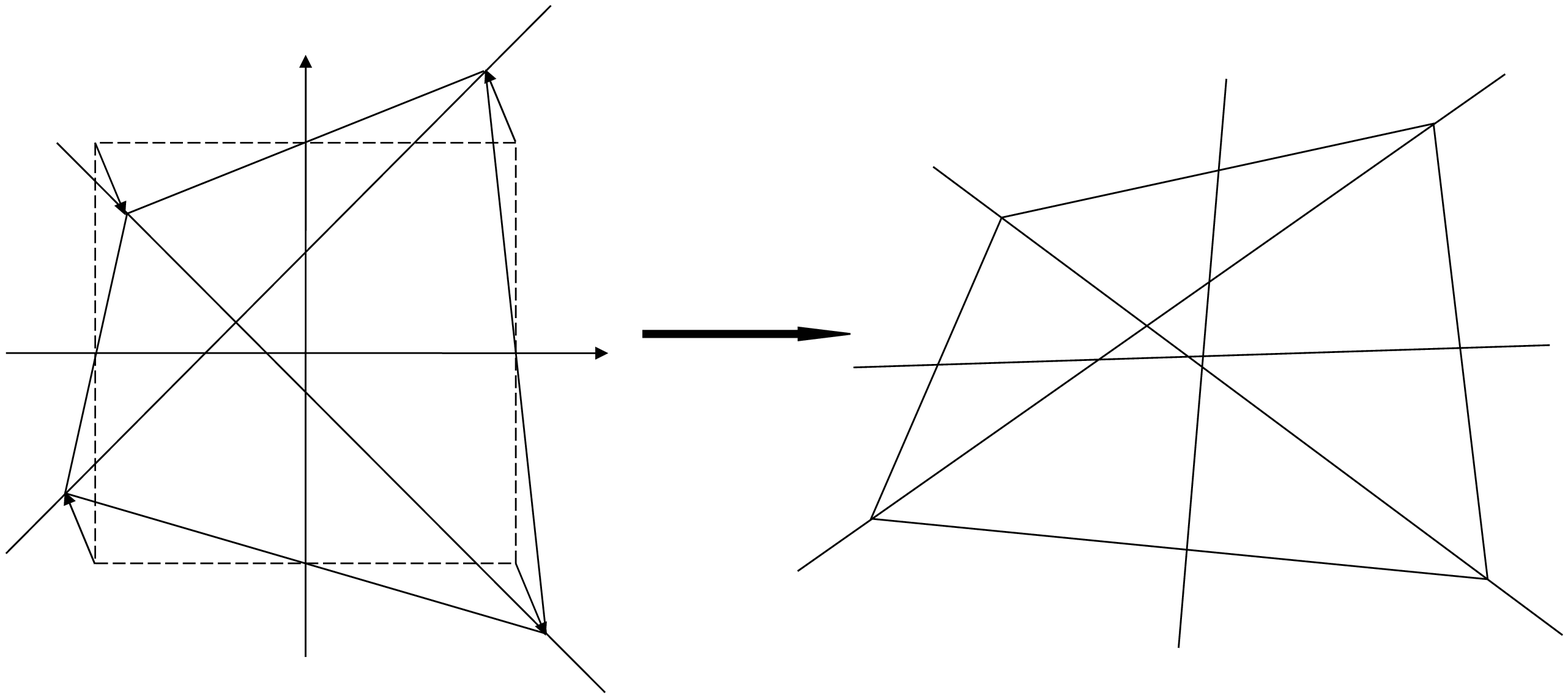}
\put(21,23){$(0,0)$} \put(40,20){$\widetilde{x}$}
\put(21.5,40){$\widetilde{y}$} \put(6,5.5){$\widehat{V}_1$}
\put(30.5,10){$\widehat{V}_2$} \put(34.5,33){$\widehat{V}_3$}
\put(6,36){$\widehat{V}_4$} \put(2.5,14){$\widetilde{V}_1$}
\put(33.5,1.5){$\widetilde{V}_2$} \put(29.5,40.5){$\widetilde{V}_3$}
\put(9,32){$\widetilde{V}_4$} \put(33.5,37){$\bm{s}$}
\put(4.5,9){$\bm{s}$} \put(16,1){$\widetilde{m}_{13}=0$}
\put(-3,23){$\widetilde{m}_{24}=0$}
\put(37,42){$\widetilde{l}_{13}=0$}
\put(40,1){$\widetilde{l}_{24}=0$}

\put(15,15){$\widetilde{K}$} \put(46,24.5){$A_K$}
\put(70,15){$K$}

\put(55.5,8.5){$V_{1}$} \put(93,5){$V_{2}$} \put(89.5,37.5){$V_{3}$}
\put(63.5,31.5){$V_{4}$} \put(77,4){$m_{13}=0$}
\put(96,19.5){$m_{24}=0$} \put(94,41){$l_{13}=0$}
\put(96,2){$l_{24}=0$}
\end{overpic}
\caption{Affine mapping $A_K$ from the intermediate
reference quadrilateral $\widetilde{K}$ to a general $K$.\label{fig: trans}}
\end{figure}

\subsection{An auxiliary 12-DoF finite element}

To design the element for fourth order singular perturbation problem,
we first introduce an auxiliary element,
which extends the rectangular Adini element to general convex quadrilaterals in a pseudo-$C^0$ manner.

\begin{definition}
\label{d: anxiliary scalar element}
The quadrilateral finite element $(K,W_K^-,T_K^-)$ is defined as follows:
\begin{itemize}
\setlength{\itemsep}{-\itemsep}
\item $K$ is a convex quadrilateral;
\item $W_K^-=P_3(K)\oplus\mathrm{span}\{\phi_1,\phi_2\}$ is the shape function space where
\[
\begin{aligned}
\phi_1&=(s_2-1)(s_2+1)l_1l_3m_{13}m_{24}-s_1s_2l_1l_3m_{24}^2+s_1l_1l_3m_{24}^2m_{13},\\
\phi_2&=(s_1-1)(s_1+1)l_2l_4m_{13}m_{24}-s_1s_2l_2l_4m_{13}^2+s_2l_2l_4m_{13}^2m_{24}.
\end{aligned}
\]
The parameters $s_1$ and $s_2$ are defined via (\ref{e: refer para}).
\item $T_K^-=\{\tau_j,~j=1,2,\ldots,12\}$ is the DoF set where
\[
\tau_j(w)=w(V_j),~(\tau_{j+4}(w),\tau_{j+8}(w))^T=\nabla w(V_j),~j=1,\ldots,4.
\]
\end{itemize}
\end{definition}

Write $p_1=l_1l_3l_4$, $p_2=l_1l_3l_2$, $p_3=l_1l_3l_{13}$, $p_4=\phi_1$
and $q_1=l_2l_4l_1$, $q_2=l_2l_4l_3$, $q_3=l_2l_4l_{24}$, $q_4=\phi_2$.
We also define the nodal functionals
\[
\begin{aligned}
\lambda_1(w)&=|E_2|\frac{\partial w}{\partial{\bm{t}_2}}(V_2),~\lambda_2(w)=|E_2|\frac{\partial{w}}{\partial{\bm{t}_2}}(V_3),~
\lambda_3(w)=|E_4|\frac{\partial{w}}{\partial{\bm{t}_4}}(V_4),~\lambda_4(w)=|E_4|\frac{\partial{w}}{\partial{\bm{t}_4}}(V_1),\\
\mu_1(w)&=|E_1|\frac{\partial w}{\partial{\bm{t}_1}}(V_1),~\mu_2(w)=|E_1|\frac{\partial{w}}{\partial{\bm{t}_1}}(V_2),~
\mu_3(w)=|E_3|\frac{\partial{w}}{\partial{\bm{t}_3}}(V_3),~\mu_4(w)=|E_3|\frac{\partial{w}}{\partial{\bm{t}_3}}(V_4)
\end{aligned}
\]
and the $4\times4$ matrices $\bm{M}$ and $\bm{N}$ by setting $\bm{M}_{i,j}=\lambda_i(p_j)$, $\bm{N}_{i,j}=\mu_i(q_j)$, $i,j=1,2,3,4$.
The following lemma is helpful to verify the unisolvency of $(K,W_K^-,T_K^-)$.

\begin{lemma}
\label{lemma: MN}
The matrices $\bm{M}$ and $\bm{N}$ are nonsingular.
\end{lemma}
\begin{proof}
For $i,j=1,2,3,4$, note that functionals
$|E_i|\frac{\partial w}{\partial{\bm{t}_i}}(V_j)=|\widetilde{E}_i|\frac{\partial \widetilde{w}}{\partial{\widetilde{\bm{t}}_i}}(\widetilde{V}_j)$,
therefore we can calculate the entries of $M$ and $N$ on $\widetilde{K}$
rather than the physical $K$ in variables $s_1$ and $s_2$.
Using (\ref{e: line equations}) we set
\begin{equation}
\label{e: f}
\begin{aligned}
f_1(s_1,s_2) &= \frac{(s_1+s_2-1)(s_1-s_2-1)}{(s_2-1)(s_2+1)},~
f_2(s_1,s_2) = \frac{(s_1-s_2+1)(s_1+s_2+1)}{(s_2-1)(s_2+1)},\\
f_3(s_1,s_2) &= \frac{(s_1+s_2+1)(s_1-s_2-1)}{(s_1-1)(s_1+1)},~
f_4(s_1,s_2) = \frac{(s_1-s_2+1)(s_1+s_2-1)}{(s_1-1)(s_1+1)}
\end{aligned}
\end{equation}
and a direct computation gives $\bm{M}=\left(\bm{M}_1^T,\bm{M}_2^T\right)^T$ and
$\bm{N}=\left(\bm{N}_1^T,\bm{N}_2^T\right)^T$, where
\[
\begin{aligned}
\bm{M}_1&=f_3\left(
\begin{array}{cccc}
(s_1 + s_2 - 1)/(s_2 - 1) & 0 & (s_1 - s_2 - 1)/(s_1 - s_2 + 1) & -(s_2+1)^2(s_1^2 + s_2 - 1)\\
(s_1 - s_2 + 1)/(s_2 - 1) & 0 & 0 & -(s_2 + 1)^2(s_1^2 + s_2 - 1)\\
\end{array}
\right),\\
\bm{M}_2&=f_4\left(
\begin{array}{cccc}
0 & (s_1 + s_2 + 1)/(s_2 + 1)& 1 & -(s_2 - 1)^2(s_1^2 - s_2 - 1)\\
0 & (s_1 - s_2 - 1)/(s_2 + 1)& 0 & -(s_2 - 1)^2(s_1^2 - s_2 - 1)\\
\end{array}
\right),\\
\bm{N}_1&=f_1\left(
\begin{array}{cccc}
0 & -(s_1 - s_2 + 1)/(s_1 + 1) & -(s_1 + s_2 - 1)/(s_1 + s_2 + 1) & (s_1 - 1)^2(- s_2^2 + s_1 + 1)\\
0 & (s_1 + s_2 + 1)/(s_1 + 1) &  0 & (s_1 - 1)^2(- s_2^2 + s_1 + 1)\\
\end{array}
\right),\\
\bm{N}_2&=f_2\left(
\begin{array}{cccc}
-(s_1 - s_2 - 1)/(s_1 - 1) &  0 & -1 &  -(s_1 + 1)^2(s_2^2 + s_1 - 1)\\
s_1 + s_2 - 1)/(s_1 - 1) &  0 &  0 &  -(s_1 + 1)^2(s_2^2 + s_1 - 1)\\
\end{array}
\right).
\end{aligned}
\]
Hence, by a symbolic computation,
\[
\begin{aligned}
\det \bm{M}&=4f_3^2f_4^2(s_1 - s_2 - 1)(s_1^2 + s_2^2 - 1),\\
\det \bm{N}&=4f_1^2f_2^2(s_1 + s_2 - 1)(s_1^2 + s_2^2 - 1).
\end{aligned}
\]
It then follows from (\ref{equation: s1+s2<1}) that $\det \bm{M}\neq0$ and $\det \bm{N}\neq0$,
which is the desired result.
\end{proof}

\begin{lemma}
\label{lemma: tangential unisol}
The element $(K,W_K^-,T_K^-)$ is well-defined.
\end{lemma}
\begin{proof}
Set $r_1=l_2l_3$, $r_2=l_3l_4$, $r_3=l_4l_1$, $r_4=l_1l_2$,
then our goal is to show $W_K^-=\mathrm{span}\{p_i,q_i,r_i,~i=1,2,3,4\}$
and that $\tau_j(w)=0$ for $w\in W_K^-$ will derive $w=0$.
First we see all $p_i$, $q_i$ and $r_i$ are linearly independent.
Indeed, if
\begin{equation}
\label{e: linear combi}
w=\sum_{i=1}^4(\alpha_ip_i+\beta_iq_i+\gamma_ir_i)=0,
\end{equation}
for $\alpha_i,\beta_i,\gamma_i\in\mathbb{R}$, then $\tau_i(w)=0$ for $i=1,2,3,4$.
Noting that
\[
\tau_i(p_j)=\tau_i(q_j)=0,~\tau_i(r_j)=r_i(V_i)\delta_{ij},~r_i(V_i)\neq0,~i,j=1,2,3,4,
\]
we obtain $\gamma_i=0$, $i=1,2,3,4$.
Moreover, the constructions of all $q_i$ imply
\[
\lambda_i(q_j)=0,~\mu_i(p_j)=0,~i,j=1,2,3,4,
\]
and therefore it follows from $\lambda_i(w)=0$ and $\mu_i(w)=0$ that
\[
\sum_{j=1}^4\alpha_j\lambda_i(p_j)=0,~\sum_{j=1}^4\beta_j\mu_i(q_j)=0.
\]
Taking $i=1,2,3,4$, we find
\[
\bm{M}(\alpha_1,\alpha_2,\alpha_3,\alpha_4)^T=\bm{N}(\beta_1,\beta_2,\beta_3,\beta_4)^T=\bm{0},
\]
which implies $\alpha_1=\alpha_2=\alpha_3=\alpha_4=\beta_1=\beta_2=\beta_3=\beta_4=0$ according to Lemma \ref{lemma: MN}.
Hence, all $p_i$, $q_i$ and $r_i$ are linearly independent,
namely, $\dim\mathrm{span}\{p_i,q_i,r_i,~i=1,2,3,4\}=\dim W_K^-=12$.
Moreover, all $p_i$, $q_i$ and $r_i$ are members of $W_K^-$,
and thus these two sets are equal.
Repeat the same process above from (\ref{e: linear combi}) apart from the assumption $w=0$,
the unisolvency is derived since $\tau_j(w)=0$ for $j=5,\ldots,12$ if and only if $\lambda_i(w)=\mu_i(w)=0$, $i=1,2,3,4$.
\end{proof}

The following lemma hints a critical property of this element,
which partly explains why the coefficients in $\phi_1$ and $\phi_2$ are necessary.

\begin{lemma}
\label{lemma: scalar tangent linear relation}
For all $w\in W_K^-$, it holds that
\begin{equation}
\label{e: scalar tangent linear relation}
\frac{1}{|E_i|}\int_{E_i}w\,\mathrm{d}s=\frac{1}{2}(w(V_i)+w(V_{i+1}))
-\frac{|E_i|}{12}\left(\frac{\partial w}{\partial \bm{t}_i}(V_{i+1})-\frac{\partial w}{\partial \bm{t}_i}(V_{i})\right),
~i=1,2,3,4.
\end{equation}
\end{lemma}
\begin{proof}
For each $i$, let $\xi_i\in P_1(E_i)$ be taken such that $\xi_i(V_i)=-1$ and $\xi_i(V_{i+1})=1$.
Then by integrating by parts, one must have
\begin{equation}
\label{e: integral by parts}
\frac{1}{|E_i|}\int_{E_i}w\,\mathrm{d}s=\frac{1}{2}\left(w(V_i)+w(V_{i+1})
-\int_{E_i}\frac{\partial w}{\partial\bm{t}_i}\xi_i\,\mathrm{d}s\right)=0,~i=1,2,3,4,~\forall w\in C^0(K).
\end{equation}
If $w\in P_3(K)$, the Simpson quadrature rule implies
\[
\frac{1}{|E_i|}\int_{E_i}\frac{\partial w}{\partial\bm{t}_i}\xi_i\,\mathrm{d}s
=\frac{1}{6}\left(\frac{\partial w}{\partial \bm{t}_i}(V_{i+1})-\frac{\partial w}{\partial \bm{t}_i}(V_{i})\right),
\]
which along with (\ref{e: integral by parts}) leads to (\ref{e: scalar tangent linear relation}).
For $\phi_1$ and $\phi_2$, we can directly calculate on $\widetilde{K}$ that
\[
\frac{1}{|\widetilde{E}_i|}\int_{\widetilde{E}_i}\widetilde{\phi}_j\,\mathrm{d}\widetilde{s}
=\frac{1}{2}(\widetilde{\phi}_j(\widetilde{V}_i)+\widetilde{\phi}_j(\widetilde{V}_{i+1}))
-\frac{|\widetilde{E}_i|}{12}\left(\frac{\partial \widetilde{\phi}_j}{\partial \widetilde{\bm{t}}_i}(\widetilde{V}_{i+1})-\frac{\partial \widetilde{\phi}_j}{\partial \widetilde{\bm{t}}_i}(\widetilde{V}_{i})\right)=0,~i=1,2,3,4,~j=1,2,
\]
and so the assertion is verified.
\end{proof}

\begin{remark}
\label{rmk: s2}
If $K$ is a rectangle, $s_1=s_2=0$.
This element degenerates to the traditional Adini element.
\end{remark}

\subsection{Enriched by bubble functions and normal aggregation}

Let us now introduce the following $H_0^1(K)$-bubble function space $W_K^b$ over $K$:
\begin{equation}
\label{e: bubble space}
W_K^b=\mbox{span}\{b_0,b_0x,b_0y,b_0l_{13}l_{24}\},~b_0=l_1l_2l_3l_4.
\end{equation}
The DoFs in $T_K^b$ with respect to these bubble functions are
\[
\tau^b_i(w)=\frac{1}{|E_{i}|}\int_{E_{i}}\frac{\partial w}{\partial \bm{n}_{i}}\,\mathrm{d}s,~i=1,2,3,4.
\]

\begin{lemma}
\label{lemma: bubble unisol}
If $w\in W_K^b$ such that all $\tau_i^b(w)=0$, $i=1,2,3,4$, then $w=0$.
\end{lemma}
\begin{proof}
Over $K$ we define the following polynomials $b_1=1$, $b_2=m_{13}$, $b_3=m_{24}$, $b_4=l_{13}l_{24}$
and $b_{i,j}=b_0b_j/l_i$ for $i,j=1,2,3,4$, then we turn to the intermediate reference quadrilateral $\widetilde{K}$.
Consider the $4\times4$ matrix $\bm{B}^-$ determined by
\[
\bm{B}^-_{i,j}=\frac{1}{|E_i|}\int_{E_i}b_{i,j}\,\mathrm{d}s
=\frac{1}{|\widetilde{E}_i|}\int_{\widetilde{E}_i}\widetilde{b}_{i,j}\,\mathrm{d}\widetilde{s},
~i,j=1,2,3,4.
\]
By (\ref{e: line equations}) and a direct computation,
we find $\bm{B}^-=\left((\bm{B}^-_1)^T,(\bm{B}^-_2)^T,
(\bm{B}^-_3)^T,(\bm{B}^-_4)^T\right)^T$,
where
\[
\begin{aligned}
\bm{B}^-_1&=\frac{1}{6}f_1\left( -1, \frac{s_2(s_1 - 1)}{5(s_1 + 1)}, \frac{s_2^2 + 5s_1 + 5}{5(s_1 + 1)},
-\frac{(s_1 + s_2 - 1)(s_1 - s_2 - 1)}{5(s_1 + s_2 + 1)(s_1 - s_2 + 1)}\right),\\
\bm{B}^-_2&=\frac{1}{6}f_3\left( 1, -\frac{s_1^2 - 5s_2 + 5}{5(s_2 - 1)},
-\frac{s_1(s_2 + 1)}{5(s_2 - 1)}, \frac{s_1 - s_2 - 1}{5(s_1 - s_2 + 1)}\right),\\
\bm{B}^-_3&=\frac{1}{6}f_2\left(-1, \frac{s_2(s_1 + 1)}{5(s_1 - 1)}, 
-\frac{- s_2^2 + 5s_1 - 5}{5(s_1 - 1)}, -\frac{1}{5}\right),\\
\bm{B}^-_4&=\frac{1}{6}f_4\left(1, -\frac{s_1^2 + 5s_2 + 5}{5(s_2 + 1)}, -\frac{s_1(s_2 - 1)}{5(s_2 + 1)},
\frac{s_1 + s_2 - 1}{5(s_1 + s_2 + 1)}\right),\\
\end{aligned}
\]
and $f_i$, $i=1,2,3,4$ have been given in (\ref{e: f}).
A symbolic computation gives
\[
\det\bm{B}^-=
\frac{f_1f_2f_3f_4((s_1^6+s_2^6)-s_1^2s_2^2(s_1^2+s_2^2)+9(s_1^4+s_2^4)-26s_1^2s_2^2+15(s_1^2+s_2^2)-25)}
{20250(s_1 - 1)(s_1 + 1)(s_2 - 1)(s_2 + 1)(s_1 + s_2 + 1)(s_1 - s_2 + 1)}.
\]
Note from (\ref{equation: s1+s2<1}) that the factor of the numerator
\[
\begin{aligned}
&(s_1^6+s_2^6)-s_1^2s_2^2(s_1^2+s_2^2)+9(s_1^4+s_2^4)-26s_1^2s_2^2+15(s_1^2+s_2^2)-25\\
\leq&(s_1^6+s_2^6)+9(s_1^4+s_2^4)+15(s_1^2+s_2^2)-25<0
\end{aligned}
\]
and therefore $\bm{B}^-$ is nonsingular.

Next, we turn to the matrix $\bm{B}$ with $\bm{B}_{i,j}=\tau^b_i(b_0b_j)$.
Our aim is to show $\bm{B}$ is nonsingular as $W_K^b=\mathrm{span}\{b_0b_j,~j=1,2,3,4\}$.
For each $i$ and $j$, since
\[
\bm{B}_{i,j}=\frac{1}{|E_i|}\left(\int_{E_i}\frac{\partial l_i}{\partial \bm{n}_i}b_{i,j}\,\mathrm{d}s
+\int_{E_i}\frac{\partial (b_{i,j})}{\partial \bm{n}_i}l_i\,\mathrm{d}s\right)
=\frac{\partial l_i}{\partial\bm{n}_i}\bm{B}_{i,j}^-,
\]
then
\[
\mbox{det}\bm{B}=\left(\prod_{i=1}^{4}\frac{\partial l_i}{\partial \bm{n}_i}\right)\mbox{det}\bm{B}^-\neq0,
\]
which completes the proof.
\end{proof}

We are in a position to propose the element for fourth order elliptic singular perturbation problems
by the normal aggregation strategy.
The DoFs are at vertices over a quadrilateral, which is a nodal type construction.

\begin{definition}
\label{d: scalar element}
The quadrilateral finite element $(K,W_K,T_K)$ is defined by:
\begin{itemize}
\setlength{\itemsep}{-\itemsep}
\item $K$ is a convex quadrilateral;
\item $W_K$ is the shape function space:
\begin{equation}
\label{e: W_K}
\begin{aligned}
W_K=\left\{w\in W_K^-\oplus W_K^b:
~\frac{1}{|E_{i}|}\int_{E_{i}}\frac{\partial w}{\partial \bm{n}_{i}}\,\mathrm{d}s
=\frac{1}{2}\left(\frac{\partial w}{\partial \bm{n}_{i}}(V_i)+\frac{\partial w}{\partial \bm{n}_{i}}(V_{i+1})\right)
,~i=1,2,3,4\right\};
\end{aligned}
\end{equation}
\item $T_K=T_K^-$ is the DoF set.
\end{itemize}
Here $W_K^b$ and $(K,W_K^-,T_K^-)$ have been given in (\ref{e: bubble space}) and Definition \ref{d: anxiliary scalar element}, respectively.
\end{definition}

\begin{theorem}
\label{th: scalar unisol}
The element $(K,W_K,T_K)$ is well-defined.
Moreover, (\ref{e: scalar tangent linear relation}) holds for all $w\in W_K$ and $P_2(K)\subset W_K$.
\end{theorem}
\begin{proof}
The four relations in (\ref{e: W_K}) hint that $\dim W_K\geq12$.
It suffices to show if $w\in W_K$ fulfilling $\tau_j(w)=0$ for $j=1,2,\ldots,12$ then $w=0$.
Write $w=w^-+w^b$ with $w^-\in W_K^-$ and $w^b\in W_K^b$.
The assumption above gives $\tau_j(w^-)=0$ as $\tau_j(w^b)=0$ for all $j$,
and by Lemma \ref{lemma: tangential unisol} we find $w^-=0$.
Moreover, the relations in (\ref{e: W_K}) will lead to $\tau^b_i(w)=\tau^b_i(w^b)=0$, $i=1,2,3,4$,
and by Lemma \ref{lemma: bubble unisol} we get $w^b=0$, which means $w=0$ and the unisolvency has been derived.
The next assertion is trivial as $w^b$ are bubble functions.
By noting that the relations in (\ref{e: W_K}) hold for $P_2(K)$,
the last assertion immediately follows.
\end{proof}

The nodal basis representation can be obtained as follows.
Let $\psi_j^-\in W_K^-$, $j=1,2,\ldots,12$ be the nodal basis of $(K,W_K^-,T_K^-)$
and $\psi_j^b\in W_K^b$, $j=1,2,3,4$ satisfy $\tau^b_i(\psi_j^b)=\delta_{ij}$, $i=1,2,3,4$.
Then
\[
\psi_j=\psi_j^-+\sum_{i=1}^4c_{i,j}\psi_i^b,~j=1,2,\ldots,12
\]
will be the nodal basis functions with respect to $(K,W_K,T_K)$,
where the coefficients $c_{i,j}$ are determined through
\[
c_{i,j}=\frac{1}{2}\left(\frac{\partial \psi_j^-}{\partial \bm{n}_{i}}(V_i)+\frac{\partial \psi_j^-}{\partial \bm{n}_{i}}(V_{i+1})\right)-\frac{1}{|E_{i}|}\int_{E_{i}}\frac{\partial \psi_j^-}{\partial \bm{n}_{i}}\,\mathrm{d}s.
\]

\subsection{Applied to fourth order elliptic singular perturbation problems}
\label{ss: fourth}

Let $\Omega\subset\mathbb{R}^2$ be a polygonal domain and $\partial\Omega$ be its boundary.
For a given $f\in L^2(\Omega)$, the fourth order elliptic singular perturbation problem appears as:
Find $u$ such that
\begin{equation}
\label{e: model problem scalar}
\begin{aligned}
&\varepsilon^2\Delta^2u-\Delta u=f~~~&\mbox{in}~\Omega,\\
&u=\frac{\partial u}{\partial\bm{n}}=0~&\mbox{on}~\partial\Omega,
\end{aligned}
\end{equation}
where $\varepsilon$ is the singular perturbation parameter tending to zero.
A weak formulation is to find $u\in H_0^2(\Omega)$ such that
\begin{equation}
\label{e: weak form scalar}
\varepsilon^2(\nabla^2u,\nabla^2v)+(\nabla u,\nabla v)=(f,v),~\forall v\in H_0^2(\Omega).
\end{equation}

Let $\{\mathcal{T}_h\}$ be a family of quasi-uniform and shape-regular partitions
of $\Omega$ consisting of convex quadrilaterals.
For a cell $K\in\mathcal{T}_h$, $h_K$ denotes the diameter of $K$,
and so the parameter $h:=\mbox{max}_{K\in\mathcal {T}_h}h_K$.
The sets of all vertices, interior vertices, boundary vertices,
edges, interior edges and boundary edges are correspondingly denoted by $\mathcal{V}_h$,
$\mathcal{V}_h^i$, $\mathcal{V}_h^b$, $\mathcal{E}_h$, $\mathcal{E}_h^i$ and $\mathcal{E}_h^b$.
For each $E\in\mathcal{E}_h$, $\bm{n}_E$ is a fixed unit vector perpendicular to $E$
and $\bm{t}_E$ is a vector obtained by rotating $\bm{n}_E$ by ninety degree counterclockwisely.
Moreover, for $E\in\mathcal{E}_h^i$, the jump of a function $v$ across $E$ is defined as $[v]_E=v|_{K_1}-v|_{K_2}$,
where $K_1$ and $K_2$ are the cells sharing $E$ as a common edge, and $\bm{n}_E$ points from $K_1$ to $K_2$.
For $E\in\mathcal{E}_h^b$, we set $[v]_E=v|_{K}$ if $E$ is an edge of $K$.

We now define the finite element space $W_h$ by setting
\[
\begin{aligned}
W_h=\Big\{w\in & L^2(\Omega):~w|_K\in W_K,~\forall K\in\mathcal{T}_h,
~\mbox{$w$ and $\nabla w$}\\
&\mbox{are continuous at all $V\in \mathcal{V}_h^i$ and vanishes at all $V\in \mathcal{V}_h^b$}\Big\}.
\end{aligned}
\]
Then the finite element approximation of (\ref{e: weak form scalar}) is: Find $u_h\in W_h$ fulfilling
\begin{equation}
\label{e: approximation form scalar}
\varepsilon^2a_h(u_h,v_h)+b_h(u_h,v_h)=(f,v_h),~\forall v_h\in W_h,
\end{equation}
where
\[
a_h(u_h,v_h)=\sum_{K\in\mathcal{T}_h}(\nabla^2u_h,\nabla^2v_h)_K,
~b_h(u_h,v_h)=\sum_{K\in\mathcal{T}_h}(\nabla u_h,\nabla v_h)_K.
\]
Moreover, we define a discrete semi-norm by setting
\[
\normmm{v}_{\varepsilon,h}^2=\varepsilon^2|v|_{2,h}^2+|v|_{1,h}^2
~\mbox{with}~|v|_{m,h}^2=\sum_{K\in\mathcal{T}_h}|v|_{m,K}^2,~m=1,2.
\]
Clearly, owing to the definition of $W_h$, we observe that $\normmm{\cdot}_{\varepsilon,h}$ is a norm on $V_h$.
Thus, by the Lax-Milgram lemma, the problem (\ref{e: approximation form scalar}) has unique solution.

For each $K\in\mathcal{T}_h$ and $s>0$,
we define the interpolation operator $\mathcal{I}_K:H^{2+s}(K)\rightarrow W_K$
according to Theorem \ref{th: scalar unisol} such that $\tau_j(\mathcal{I}_Kw)=\tau_j(w)$, $j=1,2,\ldots,12$.
Since $\mathcal{I}_Kv=v$ for all $v\in P_2(K)$ and $\{\mathcal{T}_h\}$ is quasi-uniform and shape regular,
we find
\begin{equation}
\label{e: interp err scalar}
|w-\mathcal{I}_hw|_{j,K}\leq Ch^{k-j}|w|_{k,K},~\forall w\in H^k(K)\cap H^{2+s}(K),~j=0,1,2,~k=2,3.
\end{equation}
Then the global interpolation operator $\mathcal{I}_h:H_0^2(\Omega)\cap H^{2+s}(\Omega)\rightarrow W_h$
is set as $\mathcal{I}_h|_K=\mathcal{I}_K$.
Moreover, owing to Theorem $\ref{th: scalar unisol}$, one has
\begin{equation}
\label{e: othogonal scalar}
\int_E[w_h]_E\,\mathrm{d}s=\int_E\left[\frac{\partial w_h}{\partial \bm{n}_E}\right]_E\,\mathrm{d}s=0,~
\forall E\in\mathcal{E}_h
\end{equation}
via (\ref{e: scalar tangent linear relation}) and (\ref{e: W_K}).
The Strang lemma says
\[
\normmm{u-u_h}_{\varepsilon,h}\leq C\left(
\inf_{v_h\in W_h}\normmm{u-v_h}_{\varepsilon,h}
+\sup_{w_h\in W_h}\frac{E_{\varepsilon,h}(u,w_h)}{\normmm{w_h}_{\varepsilon,h}}\right)
\]
with the consistency error
\[
E_{\varepsilon,h}(u,w_h)=\varepsilon^2a_h(u,w_h)+b_h(u,w_h)-(f,w_h).
\]
Hence, applying the proof of Theorem 1 in Chen et al.'s work \cite{Chen2005} and invoking (\ref{e: interp err scalar}),
(\ref{e: othogonal scalar}), we get
\[
\begin{aligned}
&\inf_{v_h\in W_h}\normmm{u-v_h}_{\varepsilon,h}\leq Ch\normmm{u-\mathcal{I}_hu}_{\varepsilon,h}
\leq Ch(\varepsilon|u|_3+|u|_2),\\
&E_{\varepsilon,h}(u,w_h)\leq Ch(\varepsilon|u|_3+|u|_2+\|f\|_0)\normmm{w_h}_{\varepsilon,h},~\forall w_h\in W_h,
\end{aligned}
\]
which leads to the following convergence result.

\begin{theorem}
\label{th: converge scalar}
Let $u\in H^3(\Omega)$ and $u_h\in W_h$ be the solutions of (\ref{e: weak form scalar}) and (\ref{e: approximation form scalar}), respectively.
Then
\[
\normmm{u-u_h}_{\varepsilon,h}\leq Ch(\varepsilon|u|_3+|u|_2+\|f\|_0).
\]
\end{theorem}

From Theorem \ref{th: converge scalar},
this element ensures a linear convergence order with respect to $h$, uniformly in $\varepsilon$,
provided that $\varepsilon|u|_3$, $|u|_2$ are uniformly bounded.
However these terms might blow up when $\varepsilon$ tends to zero.
The next result, following a similar line of Theorem 4.3 in \cite{Wang2013},
guarantees a uniform convergence rate under the impact of such boundary layers.

\begin{theorem}
\label{th: uniform err scalar}
Assume $\Omega$ is a convex domain.
Let $u\in H^3(\Omega)$ and $u_h\in W_h$ be the solutions of (\ref{e: weak form scalar}) and (\ref{e: approximation form scalar}), respectively.
Then it holds the uniform error estimate
\[
\normmm{u-u_h}_{\varepsilon,h}\leq Ch^{\frac{1}{2}}\|f\|_0.
\]
\end{theorem}

\section{Finite element for Brinkman problems and the associated exact sequence}
\label{s: 1-form}
\subsection{Construction of the finite element}

Let us turn to the construction of a vector-valued nodal type finite element.
As in the scalar case,
we shall prove the unisolvency of two auxiliary elements.
For a general convex quadrilateral $K$,
the element $(K,\bm{V}_K^-,\Sigma_K^-)$ and
the $\bm{H}_0(\mathrm{div};K)$-bubble element $(K,\bm{V}_K^b,\Sigma_K^b)$ are defined through
\[
\begin{aligned}
\bm{V}_K^-&=[P_1(K)]^2\oplus\mathrm{span}\{\bm{\mathrm{curl}}\,x^3,\bm{\mathrm{curl}}\,x^2y,
\bm{\mathrm{curl}}\,xy^2,\bm{\mathrm{curl}}\,y^3,\curl\,\phi_1,\curl\,\phi_2\},\\
\Sigma_K^-&=\{\sigma_j,~j=1,2,\ldots,12\},~\bm{V}_K^b=\curl\,W_K^b,~\Sigma_K^b=\{\sigma_j^b,~j=1,2,3,4\},
\end{aligned}
\]
where $\phi_1$, $\phi_2$ have been given in Definition \ref{d: anxiliary scalar element},
and the DoFs are
\[
\sigma_j(\bm{v})=\int_{E_j}\bm{v}\cdot\bm{n}_j\,\mathrm{d}s,
~(\sigma_{j+4}(\bm{v}),\sigma_{j+8}(\bm{v}))^T=\bm{v}(V_j),
~\sigma_j^b(\bm{v})=\int_{E_j}\bm{v}\cdot\bm{t}_j\,\mathrm{d}s,
~j=1,2,3,4.
\]

\begin{lemma}
Both $(K,\bm{V}_K^-,\Sigma_K^-)$ and $(K,\bm{V}_K^b,\Sigma_K^b)$ are well-defined.
\end{lemma}
\begin{proof}
We only deal with $(K,\bm{V}_K^-,\Sigma_K^-)$ as the latter is much simpler.
Define $\bm{v}_0=(x,y)^T$, then
\begin{equation}
\label{e: vk decompose}
\bm{V}_K^-=\mathrm{span}\{\bm{v}_0\}\oplus\mathrm{span}\{\curl\,w:~w\in W_K^-\}.
\end{equation}
Suppose $\bm{v}=c\bm{v}_0+\curl\,w\in\bm{V}_K^-$ for some $w\in W_K^-$ such that $\sigma_j(\bm{v})=0$ for all $j$,
then
\begin{equation}
\label{e: sum 1 to 4}
c\sigma_j(\bm{v}_0)+\sigma_j(\curl\,w)=0,~j=1,2,\ldots,12.
\end{equation}
However, we notice that $\mathrm{div}\,\bm{v}_0=2$, and by Green's formula
\[
\sum_{i=1}^4\sigma_i(\curl\,w)=\int_K\mathrm{div}\,\curl\,w\,\mathrm{d}\bm{x}=0,
~\sum_{i=1}^4\sigma_i(\bm{v}_0)=\int_K\mathrm{div}\,\bm{v}_0\,\mathrm{d}\bm{x}=2|K|\neq0.
\]
Summing over (\ref{e: sum 1 to 4}) for $j=1,2,3,4$ gives $c=0$,
and therefore $\sigma_j(\curl\,w)=0$, $j=1,2,\ldots,12$.
Hence, it suffices to show $\curl\,w=\bm{0}$.
Indeed, we can select $w$ such that $w(V_1)=0$ without changing the value of each $\sigma_j(\curl\,w)$.
Then
\begin{equation}
\label{e: point=0}
w(V_{i+1})=w(V_i)+\int_{E_i}\frac{\partial w}{\partial\bm{t}}\,\mathrm{d}s=w(V_i)+\sigma_i(\curl\,w)=0,
~i=1,2,3.
\end{equation}
Moreover,
\begin{equation}
\label{e: grad=0}
\nabla w(V_j)=(-\sigma_{j+8}(\curl\,w),\sigma_{j+4}(\curl\,w))^T=\bm{0},~j=1,2,3,4.
\end{equation}
As a consequence of (\ref{e: point=0}),(\ref{e: grad=0}) and Lemma \ref{lemma: tangential unisol},
we find $w=0$, and so $\curl\,w=\bm{0}$, which implies $\bm{v}=\bm{0}$.
The proof is done.
\end{proof}

Parallel to Lemma \ref{e: scalar tangent linear relation},
the following fact is crucial for the convergence in the Darcy limit.

\begin{lemma}
\label{lemma: vector darcy}
For all $\bm{v}\in\bm{V}_K^-$, it holds that
\begin{equation}
\label{e: vector tangent linear relation}
\frac{1}{|E_i|}\int_{E_i}\bm{v}\cdot\bm{n}\xi_i\,\mathrm{d}s=
\frac{1}{6}\left(\bm{v}(V_{i+1})-\bm{v}(V_{i})\right)\cdot\bm{n},
~i=1,2,3,4,
\end{equation}
where $\xi_i\in P_1(E_i)$ has been defined in the proof of Lemma \ref{lemma: scalar tangent linear relation}.
\end{lemma}
\begin{proof}
According to (\ref{e: vk decompose}), we shall verify this relation for $\bm{v}_0$ and all $\curl\,w$, $w\in W_K^-$.
Since $\bm{v}_0\in [P_1(K)]^2$, the Simpson quadrature rule ensures (\ref{e: vector tangent linear relation}).
On the other hand, if $w\in W_K^-$,
substituting (\ref{e: scalar tangent linear relation}) into (\ref{e: integral by parts})
will derive (\ref{e: vector tangent linear relation}) for $\bm{v}=\curl\,w$, which completes the proof.
\end{proof}

Now we introduce the nodal type vector-valued element for Brinkman problems.

\begin{definition}
\label{d: vector element}
The finite element $(K,\bm{V}_K,\Sigma_K)$ is determined through:
\begin{itemize}
\setlength{\itemsep}{-\itemsep}
\item $K$ is a convex quadrilateral;
\item $\bm{V}_K$ is the shape function space:
\begin{equation}
\label{e: V_K}
\begin{aligned}
\bm{V}_K=\left\{\bm{v}\in \bm{V}_K^-\oplus \bm{V}_K^b:
~\frac{1}{|E_{i}|}\int_{E_{i}}\bm{v}\cdot\bm{t}_i\,\mathrm{d}s
=\frac{1}{2}\left(\bm{v}(V_i)+\bm{v}(V_{i+1})\right)\cdot\bm{t}_i,~i=1,2,3,4\right\};
\end{aligned}
\end{equation}
\item $\Sigma_K=\Sigma_K^-$ is the DoF set.
\end{itemize}
\end{definition}

\begin{theorem}
\label{th: vector unisol}
The element $(K,\bm{V}_K,\Sigma_K)$ is well-defined.
Moreover, (\ref{e: vector tangent linear relation}) holds for all $\bm{v}\in \bm{V}_K$ and $[P_1(K)]^2\subset \bm{V}_K$.
\end{theorem}
\begin{proof}
The proof is very similar to that of Theorem \ref{th: scalar unisol} and thus omitted.
\end{proof}

\subsection{Applied to Brinkman problems}

Consider the following Brinkman problem of porous media flow over $\Omega$:
For given $\bm{f}\in [L^2(\Omega)]^2$ and $g\in L_0^2(\Omega)$,
find the velocity $\bm{u}$ and the pressure $p$ satisfying
\begin{equation}
\label{e: model problem}
\begin{aligned}
-\mathrm{div}\,(\nu\nabla\bm{u})+\alpha\bm{u}+\nabla p&=\bm{f}~~~~\mbox{in }\Omega,\\
\mathrm{div}\,\bm{u}&=g\,~~~~\mbox{in }\Omega,\\
\bm{u}&=\bm{0}~~~~\mbox{on }\partial\Omega.
\end{aligned}
\end{equation}
Here we assume that parameters $\nu,\alpha\geq0$ are constants but $\nu\alpha\neq0$.
A weak formulation of (\ref{e: model problem}) is to find $(\bm{u},p)\in [H_0^1(\Omega)]^2\times L_0^2(\Omega)$
satisfying
\begin{equation}
\label{equation: continuous variational formulation}
\begin{aligned}
a(\bm{u},\bm{v})-b(\bm{v},p)&=(\bm{f},\bm{v}),~\forall \bm{v}\in[H_0^1(\Omega)]^2,\\
b(\bm{u},q)&=(g,q),~~\forall q\in L_0^2(\Omega),
\end{aligned}
\end{equation}
with the bilinear forms
\[
a(\bm{u},\bm{v})=\nu(\nabla\bm{u},\nabla\bm{v})+\alpha(\bm{u},\bm{v}),
~b(\bm{v},q)=(\mathrm{div}\,\bm{v},q).
\]
This problem has a unique solution due to the following inf-sup condition
\begin{equation}
\label{e: inf-sup}
\sup_{\bm{v}\in[H_0^1(\Omega)]^2}\frac{b(\bm{v},q)}{\|\bm{v}\|_1}\geq C\|q\|_0,~\forall q\in L_0^2(\Omega)
\end{equation}
according to \cite{Boffi2013} for all possible $\nu$ and $\alpha$.

Let $\{\mathcal{T}_h\}$ be given as in Subsection \ref{ss: fourth}.
We select the following finite element spaces $\bm{V}_h$ and $P_h$:
\[
\begin{aligned}
\bm{V}_h&=\Big\{\bm{v}\in [L^2(\Omega)]^2:~\bm{v}|_K\in\bm{V}_K,~\forall K\in\mathcal{T}_h,
~\int_E[\bm{v}\cdot\bm{n}_E]_E\,\mathrm{d}s=0\mbox{ for all $E\in \mathcal{E}_h$},\\
&~~~~~~~~~~~~~~~
~~~~~~~~~~\mbox{and $\bm{v}$ is continuous at all $V\in \mathcal{V}_h^i$ and vanishes at all $V\in \mathcal{V}_h^b$}\Big\}.\\
P_h&=\left\{q\in L_0^2(\Omega):~q|_K\in P_0(K),~\forall K\in\mathcal{T}_h\right\}.
\end{aligned}
\]
If we write $\mathrm{div}_h|_K=\mathrm{div}$ on $K$,
then we have the divergence-free condition $\mathrm{div}_h\,\bm{V}_h\subset P_h$.
A discrete formulation of (\ref{equation: continuous variational formulation}) will be given as:
Find $(\bm{u}_h,p_h)\in \bm{V}_h\times P_h$, such that
\begin{equation}
\label{equation: discrete variational formulation}
\begin{aligned}
a_h(\bm{u}_h,\bm{v}_h)-b_h(\bm{v}_h,p_h)&=(\bm{f},\bm{v}_h),~\forall \bm{v}_h\in\bm{V}_h,\\
b_h(\bm{u}_h,q_h)&=(g,q_h),~~\forall q_h\in P_h,
\end{aligned}
\end{equation}
where $a_h(\cdot,\cdot)$ and $b_h(\cdot,\cdot)$ are discrete versions of $a(\cdot,\cdot)$ and $b(\cdot,\cdot)$, respectively:
\[
a_h(\bm{u}_h,\bm{v}_h)=\nu\sum_{K\in\mathcal{T}_h}(\nabla\bm{u}_h,\nabla\bm{v}_h)_K+\alpha(\bm{u}_h,\bm{v}_h),
~b_h(\bm{v}_h,q_h)=(\mathrm{div}_h\,\bm{v}_h,q_h).
\]
Moreover, the norm $\|\cdot\|_{1,h}$ and the semi-norms $|\cdot|_{1,h}$, $\|\cdot\|_{a_h}$ are equipped by
\[
\|\bm{v}_h\|_{1,h}^2=\sum_{K\in\mathcal{T}_h}\|\bm{v}_h\|_{1,K}^2,
~|\bm{v}_h|_{1,h}^2=\sum_{K\in\mathcal{T}_h}|\bm{v}_h|_{1,K}^2,
~\|\bm{v}_h\|_{a_h}^2=a_h(\bm{v}_h,\bm{v}_h).
\]
Clearly, $\|\cdot\|_{a_h}$ is a norm on $\bm{V}_h$.

For each $K\in\mathcal{T}_h$ and $s>0$,
the nodal interpolation operator $\bm{\Pi}_K: [H^{1+s}(K)]^2\rightarrow \bm{V}_K$
is defined via $\sigma_j(\bm{\Pi}_K\bm{v})=\sigma_j(\bm{v})$, $j=1,2,\ldots,12$.
Like the scalar case, we have from Theorem \ref{th: vector unisol} that
\begin{equation}
\label{e: local interp err vector}
|\bm{v}-\bm{\Pi}_K\bm{v}|_{j,K}\leq Ch^{k-j}|\bm{v}|_{k,K},~\forall\bm{v}\in[H^k(K)\cap H^{1+s}(K)]^2,~j=0,1,~k=1,2.
\end{equation}
The global interpolation operator $\bm{\Pi}_h:~[H_0^1(\Omega)\cap H^{1+s}(\Omega)]^2\rightarrow \bm{V}_h$
is naturally set as $\bm{\Pi}_h|_K=\bm{\Pi}_K$.
Since $\bm{\Pi}_K$ preserves normal integral on $E\subset\partial K$ for all $K$,
we find through integrating by parts that
\begin{equation}
\label{e: preserve bh}
b_h(\bm{\Pi}_h\bm{v},q_h)=b(\bm{v},q_h),
~\forall\bm{v}\in[H_0^1(\Omega)\cap H^{1+s}(\Omega)]^2,~\forall q_h\in P_h.
\end{equation}

Owing to the Scott-Zhang smoothing strategy \cite{Scott1990},
$\bm{\Pi}_h$ can be modified into $\overline{\bm{\Pi}}_h$ interpolating continuously from $[H_0^1(\Omega)]^2$ to $\bm{V}_h$.
Meanwhile, (\ref{e: preserve bh}) holds for all $\bm{v}\in[H_0^1(\Omega)]^2$ if
$\bm{\Pi}_h$ is replaced by $\overline{\bm{\Pi}}_h$.
Hence, by Fortin's trick and (\ref{e: inf-sup}), the following discrete inf-sup condition is derived:
\begin{equation}
\label{e: disc inf-sup}
\sup_{\bm{v}_h\in\bm{V}_h}\frac{b_h(\bm{v}_h,q_h)}{\|\bm{v}_h\|_{1,h}}\geq\sup_{\bm{v}\in
[H_0^1(\Omega)]^2}\frac{b_h(\overline{\bm{\Pi}}_h\bm{v},q_h)}{\|\overline{\bm{\Pi}}_h\bm{v}\|_{1,h}}
\geq\sup_{\bm{v}\in[H_0^1(\Omega)]^2}\frac{b(\bm{v},q_h)}{C\|\bm{v}\|_1}\geq C\|q_h\|_0,~\forall q_h\in P_h.
\end{equation}
Then by Theorem 3.1 in \cite{Zhou2018}, (\ref{equation: discrete variational formulation}) has a unique solution $(\bm{u}_h,p_h)\in\bm{V}_h\times P_h$, and
\begin{equation}
\label{e: abstract error}
\begin{aligned}
\|\bm{u}-\bm{u}_h\|_{a_h}&\leq C\left(\inf_{\bm{v_h}\in\bm{Z}_h(g)}\|\bm{u}-\bm{v}_h\|_{a_h}
+\sup_{\bm{w}_h\in\bm{V}_h}\frac{E_h(\bm{u},p,\bm{w}_h)}{\|\bm{w}_h\|_{a_h}}\right),\\
\|p-p_h\|_0&\leq C\left[\|p-\mathcal{P}_hp\|_0+M^{1/2}\left(\inf_{\bm{v_h}\in\bm{Z}_h(g)}\|\bm{u}-\bm{v}_h\|_{a_h}
+\sup_{\bm{w}_h\in\bm{V}_h}\frac{E_h(\bm{u},p,\bm{w}_h)}{\|\bm{w}_h\|_{a_h}}\right)\right],
\end{aligned}
\end{equation}
where $\mathcal{P}_h$ is the $L^2$-projection operator from $L_0^2(\Omega)$ to $P_h$, $M=\max\{\nu,\alpha\}$ and
\[
\begin{aligned}
\bm{Z}_h(g)&=\{\bm{v}_h\in\bm{V}_h:~b_h(\bm{v}_h,q_h)=(g,q_h),~\forall q_h\in P_h\},\\
E_h(\bm{u},p,\bm{w}_h)&=\sum_{K\in\mathcal{T}_h}\left(-\nu\int_{\partial K}\frac{\partial\bm{u}}{\partial\bm{n}}\cdot
\bm{w}_h\,\mathrm{d}s+\int_{\partial K}p\bm{w}_h\cdot\bm{n}\,\mathrm{d}s\right).
\end{aligned}
\]

Now we are in a position to estimate each term in (\ref{e: abstract error}).
To this end, let $\bm{u}\in [H^2(\Omega)\cap H_0^1(\Omega)]^2$ be the weak velocity solution of
(\ref{equation: continuous variational formulation}).
It follows from (\ref{e: preserve bh}) that $\bm{\Pi}_h\bm{u}\in\bm{Z}_h(g)$,
and therefore by (\ref{e: local interp err vector})
\begin{equation}
\label{e: appr err}
\inf_{\bm{v_h}\in\bm{Z}_h(g)}\|\bm{u}-\bm{v}_h\|_{a_h}
\leq\|\bm{u}-\bm{\Pi}_h\bm{u}\|_{a_h}
\leq Ch(\nu^{1/2}+\alpha^{1/2}h)|\bm{u}|_2.
\end{equation}
On the other hand, by (\ref{e: vector tangent linear relation}) and (\ref{e: V_K}),
Theorem \ref{th: vector unisol} ensures
\[
\int_Eq[\bm{v}\cdot\bm{n}_E]_E\,\mathrm{d}s=0,~\forall q\in P_1(E),~\int_E[\bm{v}\cdot\bm{t}_E]_E\,\mathrm{d}s=0,~\forall E\in\mathcal{E}_h.
\]
If $p\in H^2(\Omega)$, then following the spirit of the consistency error analysis in \cite{Zhou2018}, we have
\begin{equation}
\label{e: consistency err}
E_h(\bm{u},p,\bm{w}_h)\leq\left\{
\begin{array}{l}
Ch(\nu^{1/2}|\bm{u}|_2+\nu^{-1/2}h|p|_2),~\mbox{if $\nu\neq0$};\\
Ch(\nu^{1/2}|\bm{u}|_2+\alpha^{-1/2}|p|_2),~\mbox{if $\alpha\neq0$}.\\
\end{array}
\right.
\end{equation}
Substituting (\ref{e: appr err}) and (\ref{e: consistency err}) into (\ref{e: abstract error}),
we will obtain the following convergence result.

\begin{theorem}
\label{th: brinkman result}
Let $(\bm{u},p)\in\left([H_0^1(\Omega)\cap H^2(\Omega)]^2\right)\times(L_0^2(\Omega)\cap H^2(\Omega))$ be the weak solution of
(\ref{equation: continuous variational formulation}).
The discrete solution of (\ref{equation: discrete variational formulation}) is given by
$(\bm{u}_h,p_h)\in\bm{V}_h\times P_h$.
Then the following error estimates hold:
\begin{equation}
\label{e: err est}
\begin{aligned}
\|\bm{u}-\bm{u}_h\|_{a_h}&\leq Ch\left[(\nu^{1/2}+\alpha^{1/2}h)|\bm{u}|_2+\min\{C_1\nu^{-1/2}h,C_2\alpha^{-1/2}\}|p|_2\right],\\
\|p-p_h\|_0&\leq Ch\left\{|p|_1
+M^{1/2}\left[(\nu^{1/2}+\alpha^{1/2}h)|\bm{u}|_2+\min\{C_1\nu^{-1/2}h,C_2\alpha^{-1/2}\}|p|_2\right]\right\},
\end{aligned}
\end{equation}
where we set $\alpha^{-1/2}=+\infty$ if $\alpha=0$, and $\nu^{-1/2}=+\infty$ if $\nu=0$.
\end{theorem}

As the scalar case, boundary layers might appear if $\nu\rightarrow0$.
In such a Darcy limit, $|\bm{u}|_2$, $|p|_1$ and $|p|_2$ might explode.
We need a uniform convergence result instead of Theorem \ref{th: brinkman result}.
To this end, $\Omega$ is assumed to be a convex polygonal domain with vertices $\bm{x}_j$, $j=1,\ldots,N$ on $\partial\Omega$.
We also introduce the space
\[
H_+^1(\Omega)=\left\{q\in H^1(\Omega)\cap L_0^2(\Omega):
~\int_{\Omega}\frac{|q(\bm{x})|^2}{|\bm{x}-\bm{x}_j|^2}\,\mathrm{d}\bm{x}<\infty,~j=1,\ldots,N\right\}
\]
with the norm
\[
\|q\|_{1,+}^2=\|q\|_1^2+\sum_{j=1}^N\int_{\Omega}\frac{|q(\bm{x})|^2}{|\bm{x}-\bm{x}_j|^2}\,\mathrm{d}\bm{x}.
\]
The following result is an analogue counterpart of Theorem 3.3 in \cite{Zhou2018},
whose proof will be omitted.

\begin{theorem}
\label{th: uniform err}
Assume that $\Omega$ is convex, and $\alpha=1$, $\nu\leq1$ in (\ref{equation: continuous variational formulation}).
Moreover, the known terms $\bm{f}\in [H^1(\Omega)]^2$ and $g\in H_+^1(\Omega)$.
Let $(\bm{u},p)\in[H_0^1(\Omega)]^2\times L_0^2(\Omega)$ be the weak solution of
(\ref{equation: continuous variational formulation}).
The discrete solution of (\ref{equation: discrete variational formulation}) is given by
$(\bm{u}_h,p_h)\in\bm{V}_h\times P_h$.
Then we have the following uniform error estimate
\[
\|\bm{u}-\bm{u}_h\|_{a_h}+\|p-p_h\|_0\leq Ch^{1/2}\left(\|\bm{f}\|_1+\|g\|_{1,+}\right).
\]
\end{theorem}

\subsection{Finite element exact sequence}

In this section, we will see that the finite element spaces $W_h$, $\bm{V}_h$ and $P_h$ constitute a discrete de Rham complex.

\begin{theorem}
\label{th: discrete complex}
The following finite element sequence is exact.
\begin{equation}
\label{e: exact sequence}
\begin{tikzcd}[column sep=large, row sep=large]
0 \arrow{r} & W_h \arrow{r}{\curl_h}
& \bm{V}_h \arrow{r}{\mathrm{div}_h}& P_h \arrow{r}&0,
\end{tikzcd}
\end{equation}
where $\bm{\mathrm{curl}}_h|_K=\bm{\mathrm{curl}}$ on $K$.
\end{theorem}
\begin{proof}
We have shown $\mathrm{div}_h\bm{V}_h\subset P_h$.
Moreover, we know that $\mathrm{div}_h$ is surjective due to the discrete inf-sup condition (\ref{e: disc inf-sup}).
Next, we show $\curl_h W_h\subset\bm{V}_h$.
On one hand, for all $K\in\mathcal{T}_h$,
owing to the definitions of $\bm{V}_K^-$ and $\bm{V}_K^b$,
one has $\curl\,(W_K^-\oplus W_K^b)\subset\bm{V}_K^-\oplus\bm{V}_K^b$.
Furthermore, by the proof of Lemma \ref{lemma: vector darcy},
the relation (\ref{e: V_K}) holds for $\bm{v}=\curl\,w$, $\forall w\in W_K$.
Thus we find $\curl\,W_K\subset\bm{V}_K$, $\forall K\in\mathcal{T}_h$.
On the other hand, $\forall w_h\in W_h$,
the definition of $W_h$ ensures the continuous conditions in the definition of $\bm{V}_h$ for $\bm{v}_h=\curl_hw_h$,
which gives $\curl_h W_h\subset\bm{V}_h$.
To verify $\bm{\mathrm{curl}}_hW_h=\bm{Z}_h :=\{\bm{v}_h\in\bm{V}_h:~\mathrm{div}_h\bm{v}_h=0\}$,
it suffices to show the dimensions of this two spaces are the same since $\bm{\mathrm{curl}}_hW_h\subset\bm{Z}_h$.
Let $N_{\mathcal{V}}^i$, $N_{\mathcal{E}}^i$ and $N_{\mathcal{K}}$ be the numbers of
interior vertices, interior edges and cells in $\mathcal{T}_h$, respectively.
Then by using Euler's formula $N_{\mathcal{V}}^i-N_{\mathcal{E}}^i+N_{\mathcal{K}}=1$, we have
\[
\begin{aligned}
\dim\bm{Z}_h&=\dim\bm{V}_h
-\dim\left(\mbox{div}_h\bm{V}_h\right)
=\dim\bm{V}_h-\dim P_h\\
&=(2N_{\mathcal{V}}^i+N_{\mathcal{E}}^i)-(N_{\mathcal{K}}-1)=3N_{\mathcal{V}}^i=\dim W_h=\dim(\curl_hW_h),
\end{aligned}
\]
which implies the exactness of the sequence.
\end{proof}

\begin{remark}
With a Scott-Zhang smoothing trick \cite{Scott1990} acting on $\nabla H_0^2(\Omega)$,
the interpolation operator $\mathcal{I}_h$ in (\ref{e: interp err scalar}) can be modified into $\overline{\mathcal{I}}_h$
to work on the whole $H_0^2(\Omega)$ rather than $H_0^2(\Omega)\cap H^{2+s}(\Omega)$ (see e.g.~\cite{Guzman2014}).
As a consequence, we have the following commutative diagram:
\begin{equation}
\label{e: commuting}
\begin{tikzcd}[column sep=large, row sep=large]
0 \arrow{r} & H_0^2(\Omega) \arrow{r}{\curl} \arrow{d}{\overline{\mathcal{I}}_h}
& \left[H_0^1(\Omega)\right]^2 \arrow{r}{\mathrm{div}} \arrow{d}{\overline{\bm{\Pi}}_{h}} &
L_0^2(\Omega)\arrow{r} \arrow{d}{\mathcal{P}_h} &0\\
0 \arrow{r} & W_h \arrow{r}{\curl_h}
& \bm{V}_h \arrow{r}{\mathrm{div}_h}& P_h \arrow{r}&0.
\end{tikzcd}
\end{equation}
\end{remark}

\begin{remark}
We end this section by remarking that,
the exact sequence (\ref{e: exact sequence}) and commuting diagram (\ref{e: commuting})
can be adapted to a more general mesh type,
namely, mixed meshes consisting of both triangles and quadrilaterals,
in light of the pseudo-$C^0$ property of $W_h$ and the pseudo-$H(\mathrm{div})$ property of $\bm{V}_h$.
In fact, for a quadrilateral cell $K$,
we can still select $W_K$ as in (\ref{e: W_K}) and $\bm{V}_K$ as in (\ref{e: V_K}).
But if $K$ is a triangle, the modified nonconforming Zienkiewicz element space due to Wang et al.~\cite{Wang2007}
will be a successful candidate for $W_K$,
and $\bm{V}_K$ can also be obtained in a similar manner as in Definition \ref{d: vector element} from $W_K$.
Then $W_h$, $\bm{V}_h$ are formulated as before,
and analogous counterparts of the error estimates Theorems \ref{th: converge scalar}, \ref{th: uniform err scalar},
\ref{th: brinkman result} and \ref{th: uniform err} are also appropriate.
\end{remark}

\section{Numerical examples}
\label{s: numerical examples}

Some numerical examples are provided in this section.
Let the solution domain $\Omega$ be the unit square $[0,1]^2$,
where three types of convex quadrilateral meshes are considered.
As for the first type,
each mesh $\mathcal{T}_h$ is generated by an $n\times n$ uniform rectangular partition.
Figure \ref{fig: subfig: square} provides an example.
Meshes of the second type consists of uniform trapezoids,
see Figure \ref{fig: subfig: tixing}.
As shown in Figure \ref{fig: subfig: random}, the random partitions are demonstrated as well,
which are generated by stochastically deforming the first-type partitions with at most 20\%.
In the following examples,
the 16-node Gauss quadrature rule is adopted when the entries of stiffness matrices are accumulated for all meshes.

\begin{figure}[!htb]
\centering
\subfigure[A mesh with uniform rectangular partition] {
\label{fig: subfig: square}
\includegraphics[scale=0.45]{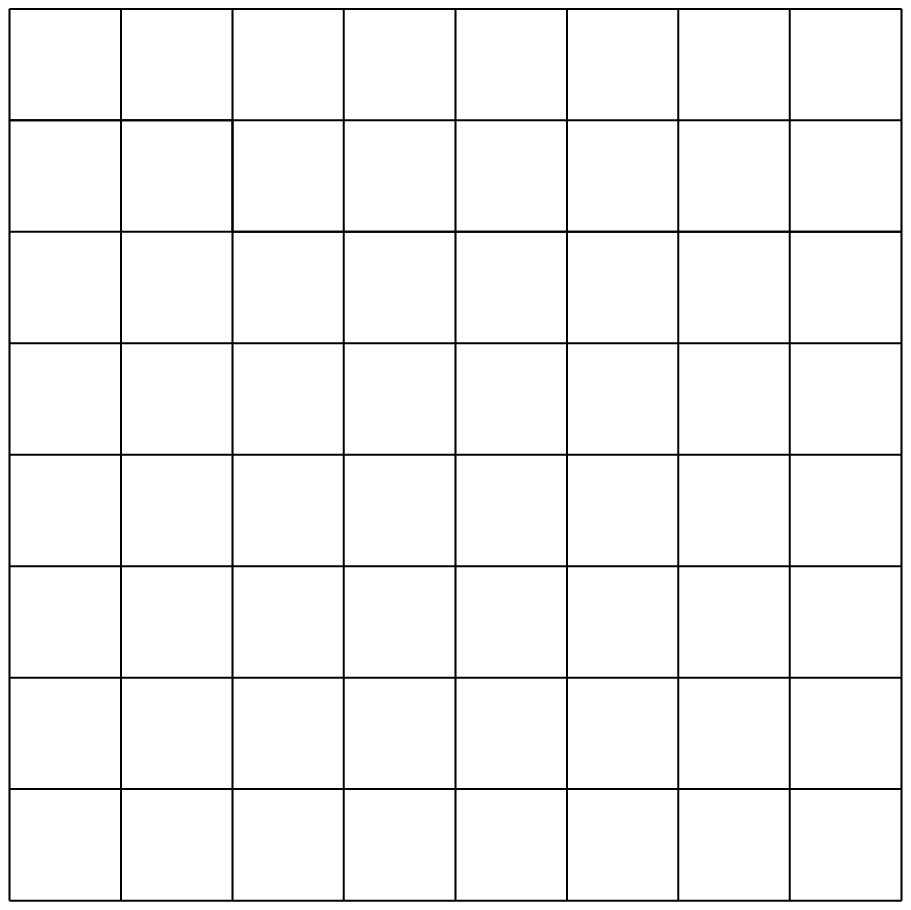}
}
\hspace{0.5cm}
\subfigure[A mesh with uniform trapezoidal partition] {
\label{fig: subfig: tixing}
\includegraphics[scale=0.45]{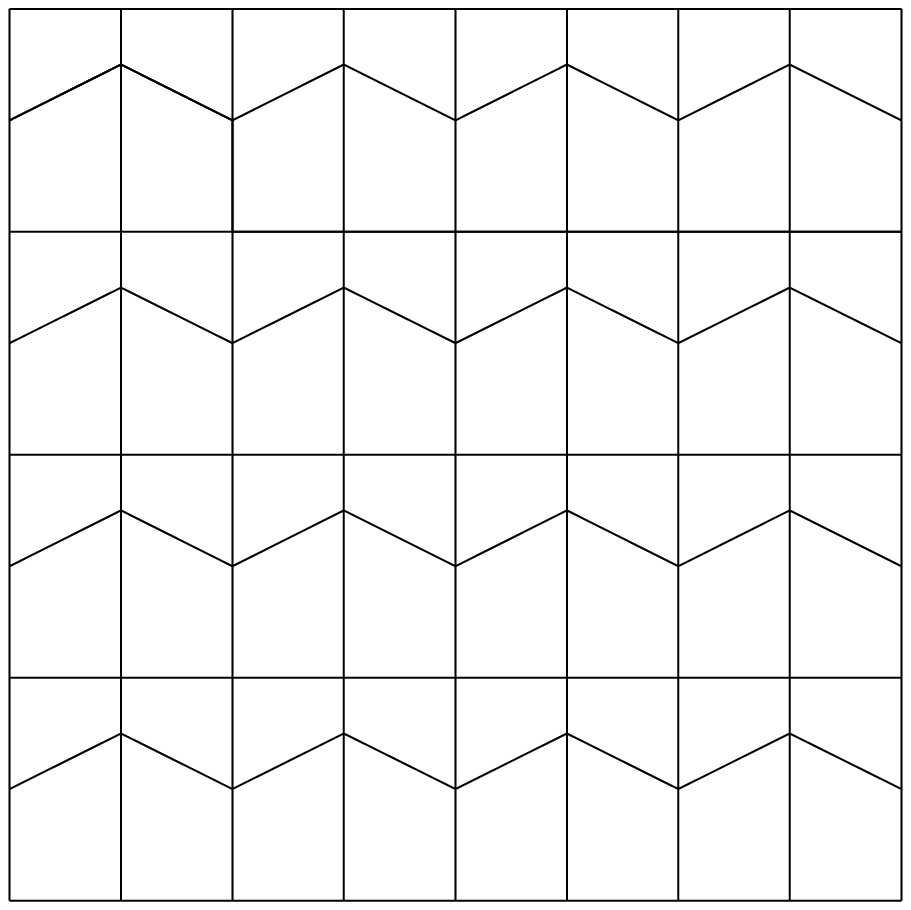}
}
\hspace{0.5cm}
\subfigure[A mesh with nonuniform randomly perturbed partition] {
\label{fig: subfig: random}
\includegraphics[scale=0.45]{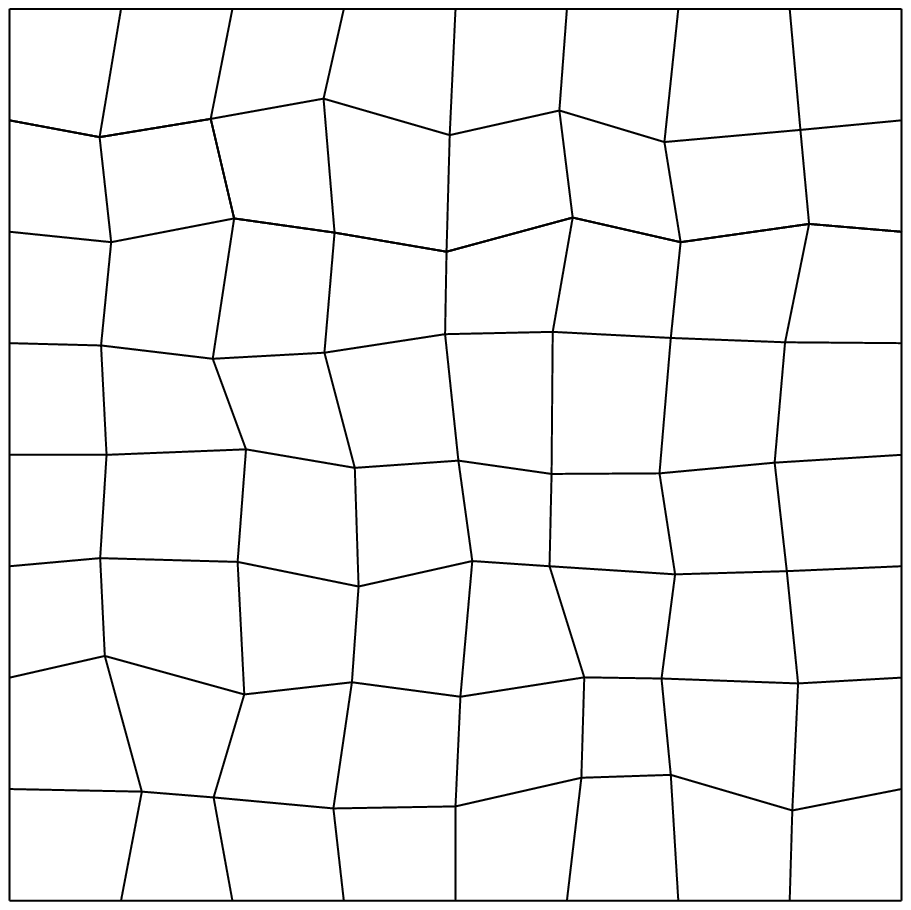}
}
\caption{Three types of quadrilateral partitions of $\Omega$.\label{fig: partitions}}
\end{figure}

Before the numerical experiment,
we give a brief analysis in terms the computational cost in comparison with some other elements working for the same problems
over the same meshes reviewed in the introduction part.
For fourth order elliptic singular perturbation problems,
we have reviewed two $H^2$-nonconforming elements in literature,
see \cite{Bao2018} for the $H^1$-conforming construction and \cite{Zhou2018} for the $H^1$-nonconforming one.
For our tested $n\times n$ meshes,
both elements are edge-based and the numbers of global DoFs are about $5n^2$.
As far as $W_h$ in this work is concerned, this number will be about $3n^2$,
reducing the computational costs in some degree benefiting from the nodal type structure.
The reduced $H^2$-conforming Fraijes de Veubeke-Sander element \cite{Ciarlet1978} has the same DoFs as ours,
but the shape function space is spline-based,
which is less preferred in practical applications than our polynomial selection.
For Brinkman problems, we investigate the $H^1$-conforming construction in \cite{Neilan2018}
and a nonconforming one in \cite{Zhou2018}.
Again, the number of global DoFs of $\bm{V}_h$ in this work is about $4n^2$,
much less than those of the two aforementioned examples:
$8n^2$ for the element in \cite{Neilan2018} and $6n^2$ for the other.

We now check the performance of the finite element space $W_h$ applied to fourth order elliptic singular perturbation problems.
The exact solution of (\ref{e: model problem scalar}) is arranged as
\begin{equation}
\label{e: example 1}
u=\sin^2(2\pi x)\sin^2(2\pi y).
\end{equation}
In Table \ref{t: example 1},
we list the errors in the energy norm $\normmm{u-u_h}_{\varepsilon,h}$ with different values of $\varepsilon$ and $h$.
The results for the biharmonic equation $\Delta^2u=f$ as well as the Poisson problem $-\Delta u=f$
with pure Dirichlet boundary conditions are also presented.
As predicted in Theorem \ref{th: converge scalar},
the first order convergence rate is observed for all possible $\varepsilon$.

\begin{table}[!htb]
\begin{center}
\begin{tabular}{p{1.5cm}<{\centering}p{1.5cm}<{\centering}p{1.5cm}<{\centering}
p{1.5cm}<{\centering}p{1.5cm}<{\centering}p{1.5cm}<{\centering}p{1.5cm}<{\centering}}
\toprule
$\varepsilon$ &  $n=4$   &  $n=8$   &  $n=16$  &  $n=32$  &  $n=64$  & order\\
\midrule
\multicolumn{4}{l}{rectangular meshes:} & \\
biharmonic & 2.909E0  & 1.315E0  & 5.913E-1 & 2.804E-1 & 1.368E-1 & 1.04 \\
$1$        & 2.913E0  & 1.315E0  & 5.914E-1 & 2.804E-1 & 1.368E-1 & 1.04 \\
$2^{-6}$   & 1.323E-1 & 3.136E-2 & 1.052E-2 & 4.537E-3 & 2.156E-3 & 1.07 \\
$2^{-12}$  & 1.236E-1 & 2.354E-2 & 5.019E-3 & 1.173E-3 & 2.866E-4 & 2.03 \\
Poisson    & 1.236E-1 & 2.354E-2 & 5.017E-3 & 1.171E-3 & 2.847E-4 & 2.04 \\
\midrule
\multicolumn{4}{l}{trapezoidal meshes:} & \\
biharmonic & 3.153E0  & 1.928E0  & 9.336E-1 & 4.562E-1 & 2.251E-1 & 1.02 \\
$1$        & 3.158E0  & 1.929E0  & 9.337E-1 & 4.563E-1 & 2.251E-1 & 1.02 \\
$2^{-6}$   & 1.161E-1 & 4.455E-2 & 1.658E-2 & 7.393E-3 & 3.551E-3 & 1.06 \\
$2^{-12}$  & 1.026E-1 & 3.177E-2 & 7.448E-3 & 1.833E-3 & 4.799E-4 & 1.93 \\
Poisson    & 1.026E-1 & 3.177E-2 & 7.444E-3 & 1.829E-3 & 4.772E-4 & 1.94 \\
\midrule
\multicolumn{4}{l}{randomly perturbed meshes:} & \\
biharmonic & 2.677E0  & 1.534E0  & 7.210E-1 & 3.553E-1 & 1.741E-1 & 1.03 \\
$1$        & 3.229E0  & 1.485E0  & 7.169E-1 & 3.563E-1 & 1.743E-1 & 1.03 \\
$2^{-6}$   & 1.286E-1 & 3.424E-2 & 1.292E-2 & 5.698E-3 & 2.749E-3 & 1.05 \\
$2^{-12}$  & 1.175E-1 & 2.662E-2 & 5.998E-3 & 1.453E-3 & 3.715E-4 & 1.97 \\
Poisson    & 1.164E-1 & 2.713E-2 & 5.971E-3 & 1.445E-3 & 3.687E-4 & 1.97 \\
\bottomrule
\end{tabular}
\caption{The errors $\normmm{u-u_h}_{\varepsilon,h}$ produced by $W_h$ applied to
 the given fourth order elliptic singular perturbation problem through (\ref{e: example 1}) over three kinds of meshes.
 \label{t: example 1}}
\end{center}
\end{table}

We then turn to the performance of the mixed finite element pair $\bm{V}_h\times P_h$
applied to the Brinkman problem.
We fix the parameter $\alpha=1$ in (\ref{e: model problem}) and test different $\nu\in(0,1]$.
The cases for the pure Darcy problem ($\nu=0$, $\alpha=1$) and Stokes problem ($\nu=1$, $\alpha=0$) are also investigated.
The exact solution of (\ref{e: model problem}) is determined by
\begin{equation}
\label{e: example 2}
\bm{u}=\curl\,(\sin^2(\pi x)\sin^2(\pi y)),~p=\sin(\pi x)-2/{\pi}.
\end{equation}
Tables \ref{t: example 2 v} and \ref{t: example 2 p} show the velocity and pressure errors, respectively.
The optimal convergence rate is also achieved for all possible parameters.

\begin{table}[!htb]
\begin{center}
\begin{tabular}{p{1.5cm}<{\centering}p{1.5cm}<{\centering}p{1.5cm}<{\centering}
p{1.5cm}<{\centering}p{1.5cm}<{\centering}p{1.5cm}<{\centering}p{1.5cm}<{\centering}}
\toprule
$\nu^{1/2}$ &  $n=4$   &  $n=8$   &  $n=16$  &  $n=32$  &  $n=64$  & order\\
\midrule
\multicolumn{4}{l}{rectangular meshes:} & \\
Stokes     & 3.186E0  & 1.503E0  & 6.926E-1 & 3.324E-1 & 1.631E-1 & 1.03 \\
$1$        & 3.190E0  & 1.503E0  & 6.927E-1 & 3.324E-1 & 1.631E-1 & 1.03 \\
$2^{-6}$   & 1.340E-1 & 3.340E-2 & 1.194E-2 & 5.327E-3 & 2.564E-3 & 1.05 \\
$2^{-12}$  & 1.236E-1 & 2.355E-2 & 5.019E-3 & 1.174E-3 & 2.874E-4 & 2.03 \\
Darcy      & 1.236E-1 & 2.354E-2 & 5.017E-3 & 1.171E-3 & 2.847E-4 & 2.04 \\
\midrule
\multicolumn{4}{l}{trapezoidal meshes:} & \\
Stokes     & 3.345E0  & 2.103E0  & 1.025E0  & 5.029E-1 & 2.486E-1 & 1.02 \\
$1$        & 3.349E0  & 2.103E0  & 1.025E0  & 5.029E-1 & 2.486E-1 & 1.02 \\
$2^{-6}$   & 1.175E-1 & 4.662E-2 & 1.789E-2 & 8.105E-3 & 3.916E-3 & 1.05 \\
$2^{-12}$  & 1.027E-1 & 3.181E-2 & 7.497E-3 & 1.884E-3 & 5.282E-4 & 1.84 \\
Darcy      & 1.026E-1 & 3.181E-2 & 7.492E-3 & 1.880E-3 & 5.257E-4 & 1.84 \\
\midrule
\multicolumn{4}{l}{randomly perturbed meshes:} & \\
Stokes     & 3.276E0  & 1.617E0  & 8.079E-1 & 4.022E-1 & 1.991E-1 & 1.01 \\
$1$        & 3.294E0  & 1.711E0  & 8.244E-1 & 3.953E-1 & 1.990E-1 & 0.99 \\
$2^{-6}$   & 1.378E-1 & 3.720E-2 & 1.392E-2 & 6.439E-3 & 3.117E-3 & 1.05 \\
$2^{-12}$  & 1.183E-1 & 2.676E-2 & 5.933E-3 & 1.492E-3 & 3.689E-4 & 2.02 \\
Darcy      & 1.317E-1 & 2.709E-2 & 6.025E-3 & 1.468E-3 & 3.778E-4 & 1.96 \\
\bottomrule
\end{tabular}
\caption{The velocity errors $\|u-u_h\|_{a_h}$ produced by $\bm{V}_h\times P_h$
applied to the Brinkman problem through (\ref{e: example 2}) over three kinds of meshes.
 \label{t: example 2 v}}
\end{center}
\end{table}

\begin{table}[!htb]
\begin{center}
\begin{tabular}{p{1.5cm}<{\centering}p{1.5cm}<{\centering}p{1.5cm}<{\centering}
p{1.5cm}<{\centering}p{1.5cm}<{\centering}p{1.5cm}<{\centering}p{1.5cm}<{\centering}}
\toprule
$\nu^{1/2}$ &  $n=4$   &  $n=8$   &  $n=16$  &  $n=32$  &  $n=64$  & order\\
\midrule
\multicolumn{4}{l}{rectangular meshes:} & \\
Stokes     & 4.593E-1 & 0.201E-1 & 5.810E-2 & 2.223E-2 & 1.027E-2 & 1.11 \\
$1$        & 4.616E-1 & 0.202E-1 & 5.827E-2 & 2.225E-2 & 1.027E-2 & 1.11 \\
$2^{-6}$   & 1.586E-1 & 7.995E-2 & 4.005E-2 & 2.003E-2 & 1.001E-2 & 1.00 \\
$2^{-12}$  & 1.586E-1 & 7.995E-2 & 4.005E-2 & 2.003E-2 & 1.001E-2 & 1.00 \\
Darcy      & 1.586E-1 & 7.995E-2 & 4.005E-2 & 2.003E-2 & 1.001E-2 & 1.00 \\
\midrule
\multicolumn{4}{l}{trapezoidal meshes:} & \\
Stokes     & 1.478E-1 & 5.967E-1 & 2.443E-1 & 1.158E-1 & 5.680E-2 & 1.03 \\
$1$        & 1.480E-1 & 5.978E-1 & 2.445E-1 & 1.158E-1 & 5.681E-2 & 1.03 \\
$2^{-6}$   & 1.569E-1 & 7.906E-2 & 3.960E-2 & 1.981E-2 & 9.907E-3 & 1.00 \\
$2^{-12}$  & 1.569E-1 & 7.906E-2 & 3.960E-2 & 1.981E-2 & 9.907E-3 & 1.00 \\
Darcy      & 1.569E-1 & 7.906E-2 & 3.960E-2 & 1.981E-2 & 9.907E-3 & 1.00 \\
\midrule
\multicolumn{4}{l}{randomly perturbed meshes:} & \\
Stokes     & 6.333E-1 & 2.721E-1 & 1.163E-1 & 5.723E-2 & 2.674E-2 & 1.10 \\
$1$        & 5.373E-1 & 3.083E-1 & 1.231E-1 & 5.165E-2 & 2.633E-2 & 0.97 \\
$2^{-6}$   & 1.620E-1 & 8.170E-2 & 4.115E-2 & 2.054E-2 & 1.030E-2 & 1.00 \\
$2^{-12}$  & 1.646E-1 & 8.191E-2 & 4.122E-2 & 2.056E-2 & 1.031E-2 & 1.00 \\
Darcy      & 1.585E-1 & 8.231E-2 & 4.119E-2 & 2.058E-2 & 1.030E-2 & 1.00 \\
\bottomrule
\end{tabular}
\caption{The pressure errors $\|p-p_h\|_0$ produced by $\bm{V}_h\times P_h$
applied to the Brinkman problem through (\ref{e: example 2}) over three kinds of meshes.
 \label{t: example 2 p}}
\end{center}
\end{table}

\end{document}